\documentclass[12pt]{article}
\usepackage{amssymb}
\usepackage{graphicx}
\usepackage{float}
\usepackage{enumerate}
\usepackage{amsmath,amssymb,amsbsy,amsfonts,amsthm,latexsym}
\usepackage{amsopn,amstext,amsxtra,euscript,amscd}

\newtheorem{theorem}{Theorem}
\newtheorem*{theorem*}{Theorem}
\newtheorem{corollary}{Corollary}
\newtheorem{lema}{Lemma}
\newtheorem{definition}{Definition}
\newtheorem{proposition}{Proposition}

\def\ln#1{\langle\langle #1\rangle\rangle}
\def\ls#1{\langle #1 \rangle}

\def\ignor#1{}

\begin{document}

\title{Almost Congruence Extension Property for subgroups of free groups}
\author{Lev Glebsky; Nevarez Nieto Saul}
\date{}
\maketitle

\begin{abstract}
Let $H$ be a subgroup of $F$ and $\left\langle \left\langle H\right\rangle \right\rangle _F$ denote 
the normal closure of $H$ in $F$. We say that $H$ has Almost Congruence Extension Property (ACEP) in $F$ if 
there is a finite set of nontrivial elements $\digamma \subset H$ such that for any normal subgroup $N$ of $H$ 
one has $H\cap \left\langle \left\langle N\right\rangle \right\rangle _F=N$ whenever $N\cap \digamma =\emptyset$.
In this paper we provide a sufficient condition for a subgroup of a free
group not possess ACEP. We also show
that any finitely generated subgroup of a free group satisfies some 
generalization of ACEP.\footnote{Part of this work was done at 
the Erwin Schr\"odinger Institute in Vienna, January-March 2016, 
during the Measured Group 
Theory program and was partially supported by 
the European Research Council (ERC) grant no. 259527 of
G. Arzhantseva. Part of this work was done at the Nizhny Nivgorod University 
and supported by the RSF (Russia) grant 14-41-00044. The authors are grateful to
the referee for useful suggestion and bringing \cite{13} to our attention.} 
\end{abstract}

\section{Introduction}
 Let $X$ be a group. We use the following standard notations:  $Y<X$ for ``Y is a subgroup of X'',
$Y\triangleleft X$ for ``$Y$ is a normal subgroup of $X$'', $\ls{Y}$ for ``the subgroup generated by $Y$, 
$\ln{Y}_X$ for ``the normal closure of $Y$ in $X$''.
(In the last two cases $Y\subseteq X$.)

\begin{definition}\label{def1}
Let $F$ be a group. A subgroup $H$ of $F$ has Congruence Extension Property (CEP) if, for any normal subgroup $N$ 
of $H$, one has 
$H\cap\left\langle \left\langle N\right\rangle \right\rangle _F=N$. 
\end{definition}

The CEP is also known by different names. Ol'shanskiĭ \cite{4} calls  it property $F(n)$, 
B. H. Newmann \cite{11} names subgroups with the CEP as E-subgroups. Stailings \cite{12} calls 
them normal convex subgroups. Finally, Osin \cite{1} introduces the self-explaining name CEP. It is worth 
mentioning that the term CEP  was used before for subalgebras and subsemigroups. 
The natural question is: when does a subgroup $H$ of a group $F$ possess CEP?
The question seems difficult, even when $H$ is a finitely generated subgroup of a free group $F$.
Particularly, its algorithmic decidability is, as far as we know, an open question.  
An obvious example of a subgroup with the CEP is a free factor of $F$, and a nontrivial example 
is given in \cite{5}. Osin \cite{1} introduced the following definition.

\begin{definition}\label{def2}
Let $F$ be a group. A subgroup $H$ of $F$ has almost CEP (ACEP) if there is a finite set of nontrivial 
elements $\digamma \subset H$ such that for any $N\vartriangleleft H$ one has 
$H\cap \left\langle \left\langle N\right\rangle \right\rangle _F=N$ whenever $N\cap \digamma =\emptyset$.
\end{definition}
Henceforth we will use $1$ to denote the identity element of the group.
An equivalent definition of ACEP is:

\begin{definition}\label{def3}
A subgroup $H$ of a group $F$ has ACEP if there is a finite set of nontrivial elements $\digamma \subset H$ such that any epimorphism $\theta$ from $H$ to 
any group $G$ can be extended to an epimorphism $\theta ^{\ast}:F\longrightarrow G^{\ast }>G$, 
whenever $\theta (\alpha )\neq 1$, $\forall\;\alpha \in \digamma$.
\end{definition}

ACEP is a natural and interesting property; moreover, it is easier to find some criteria for determining when a subgroup has ACEP. 
Our starting point is the following sufficient condition given by D. Osin in \cite{1}. 

\begin{definition} \label{def_maln}
Let $H$ be a subgroup of a group $F$, $H^{a}=\{a^{-1}ha\mid h\in H\}$ with $a\in F$. We say that $H$ is a malnormal subgroup of $F$ if $\forall\; a\in F\backslash H$
$$
H\cap H^{a}=\{1\}.
$$
We say that $H$ is cyclonormal if $\forall\; a\in F\backslash H$,
$$
H\cap H^{a}
$$
is a cyclic subgroup.
\end{definition}

\begin{theorem}[See \cite{1}]\label{1}
Any finitely generated malnormal subgroup of a free group possesses ACEP.
\end{theorem}

The theorem is Corollary 1.5 of \cite{1} and it is, indeed, a corollary of the main result 
of \cite{1}. The purpose of our paper is twofold: 
\begin{itemize} 
\item We provide a sufficient condition for a subgroup of a free group 
not possess ACEP. This condition and the theorem above allow us to "almost decide" when 
a finitely generated subgroup (f.g.s.) of a free group has ACEP. Still, we are not able to
provide an algorithm to decide when an f.g.s. of a free group has ACEP. 
\item We give further proof of Theorem~\ref{1}, different from the one provided by \cite{1}. 
Our proof uses the small cancellation theory and also leads to a more general result,
Theorem~\ref{th_acep21}. Theorem~\ref{th_acep21} is not a direct consequence of the results of
\cite{1} but may be considered as a particular case of Theorem~7.15 in \cite{13}.
Our proof gives explicit estimates of some parameters, 
see Subsection~\ref{proof_sub3}. It is not clear how to deduce these estimates from (proof of) 
Theorem~7.15 of \cite{13}.
Ultimately we show how Theorem~\ref{th_acep21} follows from results of \cite{13}.         
\end{itemize}
In the following subsections of the Introduction we describe the above mentioned items 
in more detail. Henceforth $F$ is a free group and $H$ is a finitely generated subgroup of $F$.
\subsection{Subgroups with and without ACEP.}\label{intr_sub1}
$H$ falls into one of the following categories:
\begin{enumerate}
\item $H$ is malnormal.
\item $H$ is not cyclonormal, i.e. $\exists\; a\in F\backslash H$ such that 
rank($H\cap H^{a}$)$\geq 2$.
\item $H$ is cyclonormal and $\exists\; a\in F\backslash H$ such that $H\cap H^{a}$ 
is generated by $u,$ where $u$ is not a proper power in $F$.
\item $H$ is cyclonormal but not malnormal, and $\forall\; a\in F\backslash H$ 
$H\cap H^{a}$ is either trivial or generated by a proper power.
\end{enumerate}

\begin{theorem}\label{13}
In the first case $H$, has ACEP; in the second and the third cases, $H$ does not have ACEP. 
In the fourth case, there are subgroups with ACEP and subgroups without ACEP.
\end{theorem}
Our proof is not difficult and uses that in the second and the third cases
$H$ satisfies the property:
\begin{description}
\item[(S)] 
there exists $w\in H$ such that $w^H\neq w^F\cap H$. (Here $w^X$ denotes the conjugacy
class of $w$ in $X$.)
\end{description}
We see that subgroups of free groups with the property {\bf (S)} do not have ACEP. 
In the forth case there are subgroups with and without the property {\bf (S)}. 
Then we continue the study of the property {\bf (S)} in Subsection~\ref{case4sub2}, 
where
a subgroup not having neither ACEP nor property {\bf (S)} is constructed (see Proposition \ref{example}).  
We also show the  algorithmic decidability of the property {\bf (S)}. 
   
\subsection{Theorem~\ref{1} and its generalization}\label{intr_sub2}

For $N\triangleleft H$, let $\gamma(N)=\min\{|w|,\;|\;w\in N\setminus\{1\}\}$.
Here $|w|$ denotes the length of the reduced word, representing $w\in F$. It is clear that
the following are equivalent:
\begin{itemize}
\item $H$ has ACEP. 
\item There exists $C_H>0$ such that $N=\ln{N}_F\cap H$ whenever $\gamma(N)>C_H$.
\end{itemize} 
Now, we are going to define $|\cdot |_H$, another length function on a free group $F$, 
to define $\gamma_H(N)$, and to formulate a generalization of Theorem~\ref{1}.

\begin{definition}\label{def_length}
Let $F=F(X)$ be a free group on a set $X$ and $\Omega =\{H_i\mid H_i<F,i=1,...,n\}$. 
Let $w\in F$.
We define 
$\left\vert w\right\vert _{\Omega }=\min \{k\mid w=w_1w_2\cdots w_k$, $w_i\in H_j$ or 
$w_i\in X \cup X^{-1}\}$,
i.e., $\left\vert w\right\vert _{\Omega }$ is the word metric with respect to a generating set 
$\left(\bigcup\Omega\right)\cup X\cup X^{-1}$. 
\end{definition}

We associate a special length function with a finitely 
generated $H<F$. To this end we define    
$\Omega(H)\subseteq \{H^a\cap H^b\;|\; Ha\neq Hb\}$, put $|w|_H=|w|_{\Omega(H)}$ and
$\gamma_H(N)=\min\{|w|_H\;\;|\;\;w\in N\setminus \{1\}\}$. 
We postpone the exact definition of $\Omega(H)$ to Subsection~\ref{sub_length} but mention
here two of its important properties: 
\begin{enumerate}
\item $\Omega(H)$ is finite;
\item for each nontrivial $X\in\{H^a\cap H^b\;|\; Ha\neq Hb\}$ there is $Y\in\Omega(H)$ conjugate
to $X$ in $F$.
\end{enumerate}  

\begin{theorem}\label{th_acep21}
Let $H$ be an f.g.s. of $F$. There exists 
$C_H$ such that $N=\langle\langle N\rangle\rangle_F\cap H$ whenever 
$N\triangleleft H$ and $\gamma_H(N)>C_H$.
\end{theorem}
{\bf Remarks.} We say that length functions $|\cdot|_1$ and $|\cdot|_2$ are
Lipschitz equivalent (notation: $|\cdot|_1\sim |\cdot|_2$) 
if there are $\alpha, \beta>0$ such that $\alpha |w|_1\leq |w|_2\leq\beta |w|_1$
for any $w\in F$. Let $\Omega_1$ and $\Omega_2$ be finite sets of subgroups, containing
the same nontrivial subgroups up to conjugacy. It is easy to check that 
$|\cdot|_{\Omega_1}\sim |\cdot|_{\Omega_2}$ and that Theorem~\ref{th_acep21} holds for any 
length function equivalent to $|\cdot|_H$. So, for the statement of the theorem, 
the only things that matter are the properties 1) and 2) but not the exact definition 
of $\Omega(H)$. However, the exact definition of $\Omega(H)$ matters if we want to find $C_H$ constructively. 
Clearly, $\Omega(H)\subseteq\{1\}$ for a malnormal $H$.
So, Theorem~\ref{1} is a corollary of Theorem~\ref{th_acep21}. Notice that the condition
 ``$N=\langle\langle N\rangle\rangle_F\cap H$ whenever 
$N\triangleleft H$ and $\gamma_H(N)>C_H$'' may be considered as some generalization of
ACEP. It is worth mentioning that if the normalizer of $H$ is nontrivial 
(i.e. $Norm(H)\neq H$) then $\Omega(H)$ contains a subgroup conjugate to $H$. 
So, $|\cdot|_H$ is bounded on $H$ and Theorem~\ref{th_acep21} is futile in this
case.  But it is not a weakness of the theorem. It is in the nature of things: 
$N=\langle\langle N\rangle\rangle_F\cap H$ implies that $N\triangleleft Norm(H)$.

\section{Groups without ACEP}  

\begin{lema}\label{2}
        Let $F$ be a free group and $u\in F$, 
        $u\neq 1$. 
        Then  
$\gamma(\left\langle \left\langle u^{n}\right\rangle \right\rangle _F)\geq (n-1)$.
\end{lema}

\begin{proof}
        By Theorem~3 of \cite{10} if 
$w\in \left\langle \left\langle u^{n}\right\rangle\right\rangle _F$ then $w$ 
contains a subword which is identical with a subword of
$u^n$ of length greater than $(n-1)/n$ times the length of $u^n$. 
So, $\gamma( \left\langle \left\langle u^{n}\right\rangle\right\rangle _F)\geq(n-1)|u^*|$
where $u^*$ is a cyclic reduction of  $u$.
\end{proof}

We use the following results:
\begin{theorem}\label{11}
	 
\begin{enumerate}[(a)]
	\item Two elements $u_1$ and $u_2$ of a free group $F$ have the same normal closure in $F$ if and only if $u_1$ is conjugate to $u_2$ or to $u_2^{-1}$. (Magnus \cite{2}).
	\item In a free group $F$, a nontrivial commutator cannot be a proper power. (Schützenberger \cite{7})
	\item The solutions of equation $x^ny^m=z^p$, $n,m,p>1$ in $F$ are powers of the same element, say $a$. 
Moreover, if $x=a^{k_x}$, $y=a^{k_y}$, and $z=a^{k_z}$ then $nk_x+mk_y=pk_z$. (Lyndon and Schützenberger \cite{8}).
\end{enumerate}
\end{theorem}

Let $H$ be a group, $w,a\in H$. We use the notations $w^a=a^{-1}wa$, $w^H=\{h^{-1}wh|h\in H\}$ in other words, $w^H$ is the conjugacy class of $w$ in $H$.

\begin{definition}
	Let $H<F$. We say that $H$ is an S-subgroup of $F$ (notation: $H<_SF$) 
	if $\exists\; w\in H$ such that $w^H \not= w^F \cap H$.
\end{definition}

\begin{proposition}\label{3}
Let $F$ be a free group. If $H<_SF$ then $H$ does not have ACEP in $F$.
\end{proposition}
{\bf Remark.} We will see in Proposition~\ref{example} that the converse of the statement 
is false.
\begin{proof}
Let $w\in H$ be such that $w^H \not= w^F \cap H$. 
This means that $\exists\; a\in F\backslash H$ with $ u=a^{-1}wa\in H$ such that 
$\forall\; b\in H\; u\neq b^{-1}wb\; $. Notice that $u$ cannot be conjugate to 
$w^{-1}$ in $H$. Otherwise, $w$ and $w^{-1}$ would be conjugate in a free group $F$, that is impossible for
$w\neq 1$.
By Lemma~\ref{2} 
$\gamma(\left\langle \left\langle u^{n}\right\rangle \right\rangle_H),
\gamma(\left\langle \left\langle w^{n}\right\rangle \right\rangle_H)\geq n-1$.
Now, as $u=a^{-1}wa$ we have 
$$
\left\langle \left\langle u^n\right\rangle \right\rangle _F=\left\langle \left\langle w^n\right\rangle \right\rangle _F.
$$
Also $u\neq b^{-1}wb$ and $u\neq b^{-1}w^{-1}b$ for any $b\in H$. It follows 
(by Theorem~\ref{11} (c) with $y=1$ and $n=p$) that
$
u^n \neq b^{-1}w^nb, \;
u^n \neq b^{-1}w^{-n}b.
$
Using once again Theorem \ref{11} (a) we have
$$ 
\left\langle \left\langle u^n\right\rangle \right\rangle _H\neq \left\langle \left\langle w^n\right\rangle \right\rangle _H.
$$
We deduce that $H$ does not have ACEP.
\end{proof}
\begin{lema}\label{4}
Let $a\in F\backslash H$. If $\exists\; u=a^{-1}wa\in H\cap H^a$, such that $w$ 
is not a proper power in $F$, then $H<_SF$.
\end{lema}

\begin{proof}
Suppose that $\exists\; b\in H$ such that $u=b^{-1}wb$; then 
\begin{eqnarray*}
b^{-1}wb &=&u=a^{-1}wa \\
w &=&(ab^{-1})^{-1}w(ab^{-1})
\end{eqnarray*}
therefore, $w$ commutes with $ab^{-1}$, but $F$ is a free group and this only happens if 
$w$ and $ab^{-1}$ are powers of the same element; by hypothesis $w$ is not a proper power, 
and then 
\begin{eqnarray*}
ab^{-1} &=&w^n \\
a &=&w^nb\in H,
\end{eqnarray*}
but $a\notin H$, a contradiction.
\end{proof}

\begin{corollary}\label{12}
If 
$rank(H\cap H^a)\geq 2$ for some $a\in F\setminus H$ then $H<_SF$.
\end{corollary}

\begin{proof}
Let $u,v$ be in a set of free generators of $H\cap H^a$. Notice that at least one element in 
$\{u,v,uv\}$ is not a proper power in $F$ and we are done by Lemma \ref{4}. 
Indeed, putting $x^n=u$, $y^m=v$, $z^p=uv$ implies that $\langle u, v\rangle$ is a 
cyclic group by Theorem \ref{11} (c).
\end{proof}

\begin{corollary}
Let $H$ contain a normal subgroup of $F$, i.e. 
$\exists \; N<H$ such that $N\vartriangleleft F$, then $H<_SF$.
\end{corollary}

\begin{proof}
It follows immediately from the previous corollary.
\end{proof}

\textit{Proof of the Theorem \ref{13}.} The first case follows from Theorem \ref{1}. The second and third cases follow from Lemma \ref{4}, Corollary \ref{12} and Proposition \ref{3}. The fourth case is studied in the next section.

\section{The case 4}

In this section we give examples of cyclonormal subgroups $H_{1,2}$ with 
$H_{1,2}\cap H_{1,2}^a$, 
$a\not\in H$, being trivial or generated by a proper power, such that $H_1$ does not have ACEP 
and $H_2$ has ACEP. After that we investigate the S-subgroups further 
and prove Proposition~\ref{example}.

\subsection{Examples.}
Using Proposition 9.15 \cite{3} we can see that the subgroups 
$ H_1=\left\langle x^n,y^{-1}x^ny\right\rangle $, $H_2=\left\langle x^n,y^n\right\rangle$ 
are cyclonormal and every nontrivial $H_i\cap H_i^a$,  $a\in F\backslash H$, 
is generated by a proper power. It is easy to see that  
$H_1$ does not possess ACEP while $H_2$ does. 

To prove that $H_2=\left\langle x^{n},y^{n}\right\rangle $ possesses ACEP (even CEP), 
let us take a group generated by 2 elements
$G=\ls{c,d}$ and define an epimorphism $\theta :H_2\longrightarrow G$
such that $\theta(x^n)=c$ and $\theta(y^n)=d$. Being able to 
extend  $\theta$ up to an epimorphism $\theta^{\ast }:F\longrightarrow G^{\ast }$ 
is equivalent to being able to solve the equations $x^n=c$, $y^n=d$ over $G$. 
But both equations have a solution according to 
F. Levin \cite{6}.
On the other hand, it is easy to see that $H_1$ does not possess ACEP. Moreover, it is
an S-subgroup. One can check it directly, or, alternatively, 
it follows by Proposition~\ref{6} 
below. An example of a non-S-subgroup without  ACEP is given in Proposition~\ref{example}. 
\subsection{S-Subgroups}\label{case4sub2}

\begin{definition}
Let $H<_SF$. A pair $(w,a)$ is said to be an S-witness for $H$ if $w,w^a \in H$ and 
$ \forall\;b\in H \; w^a \not= w^b$
\end{definition}
In this subsection we 
\begin{itemize}
\item Show that for a finitely generated
$H<F$ it is decidable whether $H<_SF$;
\item Give an example of a finitely generated $H<F$ that is not an S-subgroup but
does not possess ACEP.
\end{itemize}

\begin{proposition}\label{5}
Let $H<F$ and $b\in F$. Then 
\begin{itemize}
\item $H$ has ACEP if and only if $H^b$ has ACEP.
\item If $H<_SF$ with an S-witness $(w,a)$ then $H^b<_SF$ with an S-witness 
$(w^b,a^b)$.
\end{itemize}
\end{proposition}

We use the Stallings (folded) graphs, see \cite{3} for details. We start with some terminology. 
Let $F=F(X)$ be a free group on $X$. An $X$-digraph $\Gamma$ is a directed graph with edges 
labeled by elements of $X$. We denote by $V.\Gamma$ and $E.\Gamma$ the set of vertices and
the set of edges of $\Gamma$, respectively. 
For an edge $e\in E.\Gamma$ we denote the origin of $e$ by $o(e)$ and the terminus of 
$e$ by $t(e)$. The label of $e$ is denoted by $l(e)$. 
We introduce a formal inverse $e^{-1}$ of $e$ with label $l(e^{-1})=l(e)^{-1}$,
the origin $o(e^{-1}) = t(e)$, and terminus $t(e^{-1}) = o(e)$.
(Notice that $l(e)\in X$ while $l(e^{-1})\in X^{-1}$. We also suppose that 
$e^{-1}\not\in E.\Gamma$. Intuitively, if one goes in the direction 
of an edge one reads its label, say $x$, if one goes in the opposite direction of an edge 
one reads the inverse of its label, say $x^{-1}$). A path $P$ in $\Gamma$ is a sequence 
$P=e_1,...,e_k$ where $e_i\in (E.\Gamma\cup (E.\Gamma)^{-1})$  with
$o(e_i)=t(e_{i-1})$ for $1<i\leq k$. A path $P$ has a naturally defined 
label $l(P)=l(e_1),...,l(e_k)$, a naturally defined inverse path $P^{-1}=e_k^{-1},...,e_1^{-1}$ with label 
$l(P^{-1})=l(P)^{-1}$. (We suppose here that $(e^{-1})^{-1}=e$.) 
The origin and terminus of the path 
$P$ are defined as $o(P)=o(e_1)$ and $t(P)=t(e_k)$.
The label of a path is just a word in the alphabet $X\cup X^{-1}$, not necessarily 
reduced. A path in $\Gamma$ is called reduced if
it does not contains subpath $e,e^{-1}$. A cyclic path (or a cycle) is called cyclically reduced if
all its cyclic permutations are reduced. 
\begin{definition}
An $X$-digraph $\Gamma$ is called folded if for every vertex $v\in\Gamma$ and $x\in X$ there is
at most one edge $e$ with $o(e)=v$ and $l(e)=x$ as well as at most one 
edge $e$ with $t(e)=v$ and $l(e)=x$.
\end{definition}
\begin{figure}[H]
	\centering
	\includegraphics[scale=0.6]{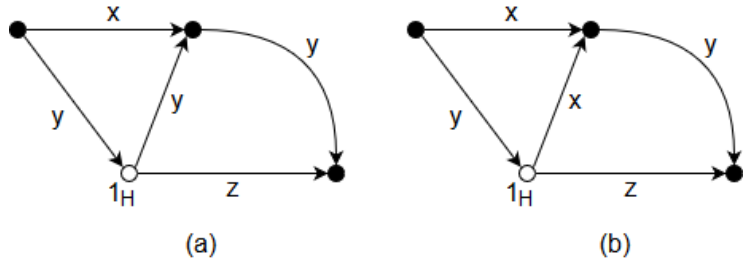}
	\caption{Folded (a) and nonfolded graphs (b).}
        \label{fig0}   
\end{figure}

Each (reduced) path in folded $\Gamma$
produces a (reduced) word in $X\cup X^{-1}$. We denote by $P(v_1,v_2,w)$ 
a path in $\Gamma$ from $v_1$ to $v_2$ with label $w$. For a folded $\Gamma$, $v\in V.\Gamma$, 
and a word $w\in (X\cup X^{-1})^*$ may exist at most one path $P(v,\cdot,w)$ 
started from  $v$ with label $w$. If $w$ is reduced, then $P(v,\cdot,w)$ is
reduced. If $P(v,v,w)$ is a cycle and $w$ is cyclically reduced, then $P(v,v,w)$
is cyclically reduced. Let $H$ be an f.g.s. of a free group $F=F(X)$.
\begin{definition}\label{def_Stallings}
The Stallings graph for a subgroup $H$ is a connected folded $X$-graph $\Gamma(H)$ with a marked vertex $1_H$
satisfying 
\begin{itemize}
\item $H$, as a set of reduced words, coincides with the set of labels of reduced cycles form $1_H$ to $1_H$;  
\item nonmarked vertex of $\Gamma(H)$ has degree at least $2$.  
\end{itemize}
Notice, that the marked vertex $1_H$ may have degree $1$.
\end{definition}
A Stallings graph $\Gamma(H)$ may be constructed using Stallings's folding, \cite{3}. The correspondence 
$H\to\Gamma(H)$ is a bijection, considering $\Gamma(H)$ up to isomorphisms 
of $X$-graphs with a marked vertex, \cite{3}. 
By definition,
$\Gamma=\Gamma(H)$ may be considered as a graph without edges of degree $1$, named $Type(\Gamma)$, 
with a ``tail'' going to $1_H$ attached. (The ``tail'' may be empty.)
Formally $Type(\Gamma)$ may be constructed as follows:
The only vertex of $\Gamma_0(H)=\Gamma(H)$ that may have degree one is $1_H$. 
In this case, removing
$1_H$ from $V.\Gamma$ we obtain the induced subgraph $\Gamma_1(H)$. The graph 
$\Gamma_1(H)$ has at most one vertex of degree 1. Removing it we obtain subgraph 
$\Gamma_2(H)$. Repeating this procedure, we end up with a graph $\Gamma_k(H)$ without 
vertices of degree 1. By definition, $Type(\Gamma(H))=\Gamma_k(H)$.     
\ignor{We need the following properties of the Stallings folded graph $\Gamma(H)$.
\begin{enumerate}
\item $H$ coincides with labels of the reduced (but not necessarily cyclically reduced) cycles
from $1_H$ to $1_H$. If $\tilde P$ is a reduction of a path $P$ then 
$l(\tilde P)$ is a reduction of $l(P)$ 
\item Let a word $\tilde w$ be a reduction of a word $w$. If there exist
a path $P=P(u,v,w)$ then the reduction of $P$ is a path $P(u,v,\tilde w)$.
\end{enumerate}
}
\begin{lema}[\cite{3}]\label{lm_type}
 Subgroups $H_1$ and $H_2$ are conjugate in $F$ if and only if
$Type(\Gamma(H_1))$ and $Type(\Gamma(H_2))$ are isomorphic as $X$-graphs (not respecting the marked vertices).
\end{lema}
\begin{proposition}\label{6}
Let $\Gamma (H)$ be the Stallings  graph for $H$. 
Then $H<_SF$ if and only if there are cyclically reduced paths $P(v,v,w)$, 
$P(v^{\prime },v^{\prime },w)$ in $\Gamma (H)$ such that one is not 
a cyclic permutation of the other.
\end{proposition}
{\bf Remark.} The existence of such a pair of paths may be effectively checked, using, for
example, the product-graph $\Gamma(H)\times\Gamma(H)$, see \cite{3} and 
Subsection~\ref{sub_length}. 

\begin{proof}
Notice that all cyclically reduced cycles lie in $Type(\Gamma(H))$. 
So, by Lemma~\ref{lm_type} and Proposition~\ref{5}, we may  
assume that $v^{\prime }=1_H$, changing $H$ by its conjugate, if necessary. 
All words in the proof are supposed to be cyclically reduced. 
So, $w_1w_2$ denotes the product in a free group, or, precisely, 
the corresponding reduced word.

$\Longleftarrow$ Suppose, that there are paths $P_2=P(v,v,w)$, 
$P_1=P(1_H,1_H,w)$ in $\Gamma (H)$ such that 
$P_2$ is not a cyclic permutation of $P(1_H,1_H,w)$. 
There is a reduced path $Q$ from $v$ to $1_H$. Let $a=l(Q)$. We show that $(w,a)$ is
an S-witness for $H$. First of all, $a\not\in H$ and the reduction of 
$Q^{-1}P_2Q$ is $P(1_H,1_H,a^{-1}wa)$. 
So, $w,a^{-1}wa\in H$. Suppose, searching for a contradiction,  
that $(w,a)$ is not an S-witness. Then there exists $b\in H$,
such that $a^{-1}wa=b^{-1}wb$.  
The reduction of $P(1_H,1_H,b)P(1_H,v,a^{-1})$ is a path $P(1_H,v,u)$ with
$u=ba^{-1}\neq 1$. 
As $[u,w]=1$ we have 
$w=y^n$ and $u=y^m$. Since $w$ is cyclically reduced, $y$ is cyclically reduced as well. 
So, product $yy$ is just a concatenation (no cancellations happen). $P(1_H,v,u=y^m)$ is a subpath of 
$P(1_H,1_H,w=y^n)$ or vice versa (by uniqueness of paths with given label and origin in folded graphs).
In either case, we deduce that $P(1_H,1_H,w)=P(1_H,v,y^k)P(v,1_H,y^{n-k})$ and 
$P_2=P(v,1_H,y^{n-k})P(1_H,v,y^k)$ is
a cyclic permutation of $P_1$, a contradiction.

$\Longrightarrow$ Suppose that $H<_SF$ and $(w,a)$ is an S-witness. W.l.g. we assume
that $w$ is cyclically reduced. It implies that either $a^{-1}w$ or $wa$ is a product 
without cancellation. Changing, if necessary, $w\to w^{-1}$, w.l.g. we may assume that 
$a^{-1}w$ is the product without cancellation. We may assume as well that not all 
letters of $w$ are canceled in $wa$. (If not, then change $a\to wa$. It may happen that after this we
have to change $w\to w^{-1}$, and, if necessary, repeat the process several times.)
As $a^{-1}wa\in H$ and the product of $a^{-1}$ and $wa$ is without cancellations paths $P(1_H,1_H,a^{-1}wa)=P(1_H,v,a^{-1})P(v,1_H,wa)$ exist.
Now, there is a path
$P(v,v,w)$ which is the reduction of $P(v,1_H,wa)P(1_H,v,a^{-1})$. 
So, there are $P_1=P(1_H,1_H,w)$ and $P_2=P(v,v,w)$. Suppose, searching for 
a contradiction, that $P_2$ is a cyclic permutation of $P_1$. In other words, 
$P_1=P(1_H,v,y)P(v,1_H,z)$ and 
$P_2=P(v,1_H,z)P(1_H,v,y)$. So, $w=yz=zy$. It implies that $y^{-1}wy=w$, $b=ya\in H$
and $b^{-1}wb=a^{-1}wa$, a contradiction.    
\end{proof}
\begin{proposition}\label{example}
Let $H=\left\langle a^4, a^2ba, aca^2,bc^{-1}\right\rangle$. 
Then $H$ is not an S-subgroup of $F(a,b,c)$ and does not have ACEP.
\end{proposition}
\begin{proof}
First of all, using the Stallings graph $\Gamma (H)$  and Lemma 6.1 \cite{3} we see 
that $H$ is freely generated by
$$
w_1=a^4,\; w_2=a^2ba,\; w_3=aca^2,\; w_4=bc^{-1}.
$$
Let $N=\left\langle \left\langle [w_1^n,w_2],[w_1^n,w_3]\right\rangle \right\rangle _H$, 
$n\in \mathbb{Z}$ and $[\cdot,\cdot]$ is a commutator. 
One can check that $\gamma(N)\geq \lfloor\frac{n-1}{2}\rfloor$ and 
$[w_1^n,w_4]\not\in N$. Indeed,
$$
H/N=(\langle u,w_2,w_3\;|\;[u,w_2],\;[u,w_3]\rangle\mathop{*}\limits_{u=w_1^n}\langle w_1\rangle)*\langle w_4\rangle,
$$
Consider the image of $\langle w_1,w_2,w_3\rangle$ in $H/N$. 
The elements of $H/N$ have the following normal form: 
$$
w_1^{rn}p_1(w_2,w_3)w_1^{\alpha_1}\dots,
$$
where $p_i\in \ls{w_2,w_3}$ and 
$-\lfloor\frac{n-1}{2}\rfloor\leq\alpha\leq\lfloor\frac{n}{2}\rfloor $, see \cite{9}.
It follows that any word in $w_1,\dots,w_4$ of length less 
than $\lfloor\frac{n-1}{2}\rfloor$ 
is in the normal form and does not belong to $N$. Also, it is easy to check that 
$[w_1^n,w_4]=w_1^nw_4w_1^{-n}w_4^{-1}\not\in N$.   
But
$[w_1^n,w_4]\in \left\langle \left\langle N\right\rangle \right\rangle _F$. Indeed, 
$w_4=a^{-2}w_2aw_3^{-1}a$ and $[w_1,a]=1$ by definitions of $w_i$.

\begin{figure}[H]
	\centering
	\includegraphics[scale=0.5]{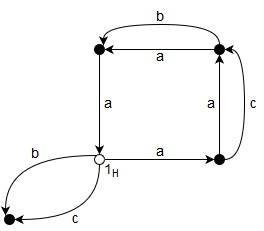}
	\caption{$\Gamma(H)$}
         \label{fig1}
\end{figure}

Using the graph $\Gamma(H)$ we can see that there are no cyclically reduced paths $P(v,v,h)$, $P(v^{\prime },v^{\prime },h) \in \Gamma (H)$ such that $P(v,v,h)$ is not a cyclic permutation of $P(v^{\prime },v^{\prime },h)$. Then by Proposition \ref{6} $H$ is not an $S$-subgroup.
\end{proof}

\section{Proof of Theorem \ref{th_acep21}}\label{sec_proof}

In this section we prove  Theorem~\ref{1} by means of the small cancellation 
theory (see \cite{2}) with the use of Stallings's folded graphs. 
Let $H$ be an f.g.s. of a free group $F$ and $N\triangleleft H$.
We start with the construction of derivation diagrams for 
$u\in\langle\langle N\rangle\rangle_F$.
\subsection{Derivation diagram for $u\in\langle\langle N\rangle\rangle_F$}

\begin{definition}
A morphism from an $X$-digraph $\Gamma _1$ to an $X$-graph $\Gamma _2$ is a map $\pi$ 
from the set of vertices and edges of $\Gamma _1$ to the set of vertices and edges of 
$\Gamma _2$ sending vertices to vertices, edges to edges, preserving
labels ($l(e)=l(\pi(e))$) and extremes of edges ($o(\pi (e))=\pi (o(e))$, 
$t(\pi (e))=\pi (t(e))$).
\end{definition}

\begin{definition}
Let $E.\Gamma$ be the set of edges of $\Gamma$.
The in-star of a vertex $v$ in a graph $\Gamma$ is 
the set $instar(v)=\{e\in E.\Gamma\;|\; v=t(e)\}$.
The out-star of a vertex $v$ in a graph $\Gamma$ is the set
$outstar(v)=\{e\in E.\Gamma\;|\; v=o(e)\}$
\end{definition}

\begin{definition}
 A graph $\Gamma _1$ is a covering graph of a graph $\Gamma _2$ if there is a 
covering map from $\Gamma _1$ to $\Gamma _2$. A covering map 
$f:\Gamma _1 \rightarrow \Gamma _2$ is an $X$-digraph  epimorphism such that 
the in-star and out-star of each vertex $v$ of $\Gamma _1$ is mapped bijectively 
onto the in-star and out-star of its image $f(v)$, respectively.  
\end{definition}
Notice that a covering graph for a folded graph is folded.
\begin{proposition}\label{8}
Let $H<F,\; N\vartriangleleft H$, and $\Gamma (H)$ be the corresponding 
Stallings folded graph for $H$. There is a connected $X$-digraph $\Gamma (N)$ that  is a covering of $\Gamma (H)$ 
with a covering map $f:\Gamma(N)\to \Gamma(H)$ satisfying the following property:
\begin{description}
\item[*)] for every $v\in f^{-1}(1_H)$ the group $N$ (as a set of reduced words) coincides with the labels of
reduced paths from $v$ to $v$.  
\end{description}  
Property {\bf*)} uniquely (up to  isomorphism) determines a connected covering of $\Gamma(H)$.
\end{proposition}

\begin{proof}
We give a precise construction for $\Gamma(N)$ as the covering graph of $\Gamma(H)$ with 
a desk transformation $H/N$. Fix a spanning tree $T$ in $\Gamma(H)$.
For $v\in V.\Gamma(H)$ let $P(v)$ be the reduced path in $T$ from $1_H$ to $v$.
For each edge $e\not\in E.T$ let $h_e$ be the label of the path $P(o(e))eP^{-1}(t(e))$.
The set $\{h_e\;|\;e\in E.\Gamma(H)\setminus E.T\}$ generates $H$ freely, \cite{3}. 
Let $G=H/N$ and $\bar h$ denote the image of $h\in H$ in $G$ by the natural homomorphism
$H\to H/N$. We define $\Gamma(N)$ as follows:
\begin{itemize}
\item $V.\Gamma(N)=V.\Gamma(H)\times G$,
\item $E.\Gamma(N)=E.\Gamma(H)\times G$,
\item $l(e,g)=l(e)$, $o(e,g)=(o(e),g)$, and
$
t(e,g)=\left\{\begin{array}{lll} (t(e),g), & \mbox{if} & e\in T\\
                                 (t(e),g\bar h_e) & \mbox{if} & e\not\in T
              \end{array}\right.   
$
\end{itemize}
It is routine to check that $N$ coincides with labels of the reduced cycles 
of $(1_H,g)$ of $\Gamma(N)$ and that map $f:\Gamma(N)\to \Gamma(H)$, $f(a,g)=a$, is a covering map. Notice that it is not difficult to construct an isomorphism 
between two connected covering graphs of $\Gamma(H)$ with property {\bf *)}. 
\end{proof}

Let $f:\Gamma_1\to \Gamma_2$ be an $X$-digraph morphism. We naturally extend $f$ to the inverses 
of edges: $f(e^{-1})=(f(e))^{-1}$. For any path $Q=e_1,e_2,\dots,e_k$
in $\Gamma_1$ its image $f(Q)=f(e_1),f(e_2),\dots,f(e_k)$ is a path in $\Gamma_2$. 
It is clear by definition that $l(Q)=l(f(Q))$. 
If $P$ is a path in $\Gamma _2$, a lift of $P$ is a path $Q$ in $\Gamma_1$ such that 
$f(Q)=P$.
\begin{lema}
Let $\Gamma_1$, $\Gamma_2$ be graphs and 
$f:\Gamma_1\longrightarrow\Gamma_2$ be a covering map. 
If $P$ is a path in $\Gamma_2$ then
for any $v\in f^{-1}(o(P))$, there is a unique lift of $P$ starting at $v$.
\end{lema}

\begin{lema}\label{lm_lev1}
	Let $N\vartriangleleft H$ and 
$f:\Gamma (N)\longrightarrow \Gamma (H)$ be a covering map.
	Then $\forall $ $v\in V.\Gamma (H)$ and $\forall v_1,v_2\in f^{-1}(v)$
	there exist an automorphism $\alpha :\Gamma (N)\longrightarrow \Gamma (N)$ such
	that $\alpha (v_1)=v_2$.
\end{lema}

\begin{proof}
Indeed, $v_1=(v,h_1)$, $v_2=(v,h_2)$,
and $\alpha(a,h)=(a,(h_2h_1^{-1})h)$.    
\end{proof}

A presentation $\ls{X\;|\;R}$ is called symmetrized if $R$ consists of cyclically 
reduced words and for any $w\in R$ all cyclic permutations of $w$ as well as $w^{-1}$
belong to $R$.
A diagram $M$ over the symmetrized presentation $\left\langle X\mid R\right\rangle $ 
is a planar finite cell complex (which we denote, abusing notation, by $M$ as well), 
given with a specific
embedding $M\subseteq \mathbb{R}^{2}$, the following additional data, and satisfying 
the following additional properties:

\begin{enumerate}
\item The complex $M$ is connected and simply connected.

\item Each edge (one-cell) of $M$ is labeled by an arrow and a letter $x\in
X$.

\item Some vertex (zero-cell) which belongs to the topological boundary $\delta M$ of 
$M\subseteq \mathbb{R}^{2}$  is specified as a base-vertex.

\item For each region (two-cell) $D$ of $M$ the label of boundary cycle is in $R$,
(in our notations $l(\delta(D))\in R$). (Notice that $l(\delta(D))$ is defined up to
cyclic permutation, if we do not specify a base vertex of $\delta(D)$.)
\end{enumerate}

Thus the 1-skeleton of $M$ is a finite connected planar graph $M^1$
embedded in $\mathbb{R}^{2}$ and the two-cells of $M$ are precisely the bounded 
complementary regions for this graph. Each geometric edge $e$ has two 
orientations. We chose the positive orientation, such that the label 
of $e$ is in $X$.
With this convention we consider $M^1$ as an $X$-digraph with $E.M^1$ being a
set of positively oriented edges.
For a two-cell $D$ its boundary 
$\delta(D)$ is a closed path $e_1,v_1,e_2,v_2...e_kv_k$ in $M^1$ with edges 
$e_j\in E.M^1\cup (E.M^1)^{-1}$
and vertices $v_j\in V.M^1$.  
A diagram $M$ also has the boundary cycle, denoted by $\delta M$, which is an
edge-path in the graph $M^1$ corresponding to going around $M$ once in
the clockwise direction along the boundary of the unbounded complementary
region of $M^1$, starting and ending at the base-vertex of $M$. The
label of that boundary cycle is a word $w$ in the alphabet $X\cup X^{-1}$
that is called the boundary label of $M$. We call $M$ a derivation diagram for 
$w$. 
It is a lemma of Van Kampen  that $w\in \langle\langle R\rangle\rangle_F$ if and only if a derivation diagram for $w$ exist, see \cite{2} page 237.
As we are interested in $\langle\langle N\rangle\rangle_F$ for $N\triangleleft H<F$ we
chose $R$ to be the set of  labels of all cyclically reduced cycles in $\Gamma(N)$.
This implies that for any region $D\in M$ there is a cycle $C_D=P(v,v,w)$ in $\Gamma(N)$ with
$w=l(\delta(D))$. It is possible that a vertex of $M^1$ appears more than once in  $\delta(D)$. 
We call such a vertex a multiple vertex of $\delta(D)$. 
Also, it is possible that both $e$ and $e^{-1}$ are in $\delta(D)$. We call such 
an edge a double edge of $\delta(D)$.  
A graph formed by multiple vertices and double edges of $\delta(D)$ is called 
the self-boundary of $D$. Notice that a connected component of a self-boundary is
a tree.  
Let $\delta(D)=e_1,v_1,e_2,v_2,\dots,e_k,v_k$ with edges $e_i$ and vertices $v_i$. Changing $\delta(D)$ by 
$\tilde\delta(D)=(1,e_1),(1,v_1),(2,e_2),(2,v_2),\dots,(k,e_k),(k,v_k)$, we get a cycle with different vertices 
and edges with labels $l((j,e_j))=l(e_j)$.  Intuitively,
we think of $\tilde\delta(D)$ as an infinitesimal shift of $\delta(D)$ inside of $D$.
In $\tilde\delta(D)$ there are edges with labels from $X$ (positive edges) and edges with labels from
$X^{-1}$ (negative edges). Changing the negative edges in $\tilde\delta(D)$ by its inverse we obtain an
$X$-digraph $\delta'(D)$. 
By construction there exists a  morphism $\phi_D:\delta'(D)\to \Gamma(N)$.   
\begin{definition} \label{def_diag}
Let $w\in \langle\langle N\rangle\rangle_F$. A derivation $N$-diagram for $w$ is
a diagram $M$ with the following properties:
	\begin{itemize}
		\item 
		$l(\delta (M))=w$, where $\delta (M)$ denotes the boundary of $M$.
		\item For each label of an internal region of $M$ there exists a 
cyclically reduced
cycle in $\Gamma(N)$ with the same label. So, for a region $D$ there exists a 
labeled graph morphism $\phi_D:\delta'(D)\to\Gamma(N)$.
	\end{itemize}
\end{definition}
The discussion above shows that $w\in\langle\langle N\rangle\rangle_F$  iff a derivation $N$-diagram for $w$ exist.
We call a path $P$ in $M^1$ a piece if degrees of $o(P)$ and $t(P)$ are at least $3$ and degrees of
all other vertices of $P$ are $2$. We call a piece $P$ external if $P\in\delta(M)$, otherwise we call it
internal. In what follows we consider internal pieces only, so ``piece'' means ``internal piece'' through the text.  
Let $P$ be a piece. There exist regions  
$D$ and $E$ such that $P\in\delta (D)$ and $P^{-1}\in\delta(E)$. 
There is a natural image $P_D$ (resp. $P_E$) of $P$ in  $\delta'(D)$  (resp. $\delta'(E)$), see fig.\ref{fig2}.
If $E=D$ then $P$ is a path in the self-boundary of $D$. In this case, there are two different paths in 
$\delta'(D)$ corresponding to $P$. 
We say that a piece
$P$ is inessential if there are morphisms 
$\phi_D:\delta'(D)\to\Gamma(N)$ and $\phi_E:\delta'(E)\to\Gamma(N)$ such that 
$f(\phi_D(P_D))=f(\phi_E(P_E))$. (Recall that 
$f:\Gamma(N)\to\Gamma(H)$ is a covering map. 
Also, if $D=E$ we assume that $\phi_D=\phi_E$.)
We say that $P$ is essential if it is not inessential.
For an $N$-diagram $M$ let $r(M)$ be a number of regions.
For $u\in\langle\langle N\rangle\rangle_F$ let 
$r(u)=\min\{r(M)\;|\;\mbox{$M$ is a derivation $N$-diagram for $u$}\}$.
A derivation $N$-diagram $M$ for $u$ is called optimal if $r(M)=r(u)$.

\begin{figure}[H]
	\centering
	\includegraphics[scale=0.5]{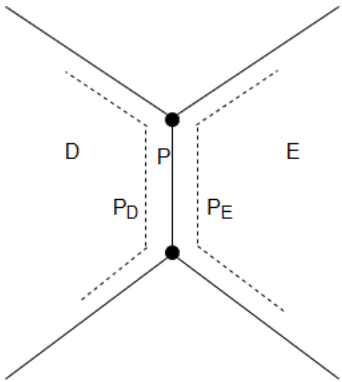}
	\caption{The piece P in the D and E regions}
	\label{fig2}
\end{figure}

\begin{lema}\label{lm_ess_diff}
An inessential piece of an 
optimal $N$-diagram is a path in a self-boundary of a region.  
\end{lema}
\begin{proof}
Let $M$ be a derivation $N$-diagram for $u$. 
Suppose, that a piece $P\subseteq \delta(D)\cap \delta(E)$ is 
inessential and $D\neq E$. We may write $\delta(D)=Q_1P$ and 
$\delta(E)=P^{-1}Q_2$. Let $D'$ be a region, obtained by gluing $D$ and $E$ along
$P$. Construct $M'$ by removing $D\cup E$ and gluing $D'$ along $Q_1Q_2$. 
If $Q_1Q_2$ is not cyclically reduced we may apply folding to $M'$, see 
section 5.1 of \cite{2}. Clearly, the resulting diagram has $r(M)-1$ region.
To get a contradiction it suffices to show that it   
is an $N$-diagram or, the same, that 
in $\Gamma(N)$ there exists a cycle with label $l(Q_1Q_2)$. But 
$f(\phi_D(P))=f(\phi_E(P))$ for some $\phi_D$, $\phi_E$. By Lemma~\ref{lm_lev1} there is an automorphism
$\alpha:\Gamma(N)\to\Gamma(N)$ such that $\alpha(\phi_D(P))=\phi_E(P)$.
So, $\alpha(\phi_D(Q_1))\phi_E(Q_2)$ is a required cycle .  
\end{proof}
Let $P$ be an 
inessential piece in the self-boundary of $D$, i.e. 
$\delta(D)=Q_1PQ_2P^{-1}$. $P$ splits in $\delta'(D)$ in two paths, say 
$P_1$ and $P_2$. It may happen that $\phi_D(P_1)\neq \phi_D(P_2)$ 
for a morphism $\phi_D:\delta'(D)\to\Gamma(N)$.
By definition of an inessential boundary component $f(\phi_D(P_1)=f(\phi_D(P_2)$ 
for some morphism $\phi_D$. (Particularly, it implies that
$f(\phi_D(Q_i))$ are cycles in $\Gamma(H)$.) 

\subsection{Length functions}\label{sub_length}

In this subsection we give a detailed definition of $|\cdot|_H$.
Recall the definition of $|\cdot|_\Omega$ in Subsection~\ref{intr_sub2}.
Recall, as well, that $|\cdot|_H=|\cdot|_{\Omega(H)}$. So, the only thing
we need is a detailed definition of $\Omega(H)$. We start with some properties of
$\Gamma(H)$.   

\begin{proposition}\label{prop_paths}
Let $v_i\in\Gamma(H)$ and $P(1_H,v_i,w_i)$ be reduced paths in $\Gamma(H)$ for some
$w_i\in F$, $i=1,2$. 
Then $w_1w_2^{-1}\in H$ if and only if $v_1=v_2$.
\end{proposition} 
\begin{proof}
The proof uses the uniqueness in $\Gamma(H)$ of paths with given origin and label.

$\Longrightarrow$ Let $w_1=w'_1\alpha$ and $w_2=w'_2\alpha$ be such that $w'_1(w'_2)^{-1}$ 
and $w'_i\alpha$ are products without cancellation. As $w'_1(w'_2)^{-1}=w_1w_2^{-1}\in H$,
there is a cycle $P(1_H,1_H,w'_1w_2^{\prime-1})=P(1_H,v',w'_1)P(v',1_H,w_2^{\prime-1})$.
Now,  
$$
P(1_H,v_1,w_1)=P(1_H,v',w'_1)P(v',v_1,\alpha) \mbox{ and } 
P(1_H,v_2,w_2)=P(1_H,v',w'_2)P(v',v_2,\alpha).
$$ 
So, $v_1=v_2$ by the uniqueness of the path $P(v',\cdot,\alpha)$.

$\Longleftarrow$ Similar.
\end{proof}

Consider a path $P(1_H,v,w^{-1})$ in $\Gamma(H)$. By Proposition~\ref{prop_paths},
$H^w$ depends on $v$ only. So, for $v\in V.\Gamma(H)$ we define 
$H^v=H^w$ for some $w\in F$ with $P(1_H,v,w^{-1})$. Actually, $H^v$ consists of
labels of the reduced cycles from $v$ to $v$ in $\Gamma(H)$.  
Now, let
$\Omega(H)=\{H^a\cap H^b\;|\;a,b\in V.\Gamma(H) \;\land\;a\neq b\}$.
We define $|w|_H=|w|_{\Omega(H)}$. 

\begin{definition} \label{def_diam} 
We define $diam(\Gamma )=\max \{d(u,v)\;|\;\ u,v\in V.\Gamma,\; d(u,v)<\infty \}$, 
where $d(u,v)$ is 
a graph distance, that is, the length of the shortest path from $u$ to $v$. If 
$P$ is a path in $\Gamma$, we denote by 
$\left\vert P\right\vert$ the length of $P$. Notice that $diam(\Gamma)$ is
finite for a finite graph $\Gamma$ (even disconnected).
\end{definition}

Consider $\Gamma(H)\times \Gamma(H)$, the product of $\Gamma(H)$ in the category of 
$X$-labeled graphs. That is, $V.(\Gamma\times\Gamma)=(V.\Gamma)\times (V.\Gamma)$
and $l^{-1}_{\Gamma\times\Gamma}(x)=l^{-1}_{\Gamma}(x)\times l^{-1}_\Gamma(x)$. Where $l^{-1}_G(x)$
is the set of edges with label $x$ in a graph $G$.  Notice that 
 path $P((v_1,v'_1),(v_2,v'_2),w)$ in $\Gamma\times\Gamma$ defines two paths $P(v_1,v_2,w)$ and
$P(v'_1,v'_2,w)$ in $\Gamma$. It follows that $H^v\cap H^u$ consists of 
the labels of cycles from $(v,u)$ to $(v,u)$ in $\Gamma(H)\times\Gamma(H)$.  
The next proposition follows from the properties of folded graphs.
\begin{proposition}
Suppose, that there is an edge from $(u,v)$ to $(u',v')$ in $\Gamma(H)\times\Gamma(H)$.
Then $u=v$ if and only if $u'=v'$.
\end{proposition}

So, $\Gamma(H)\times\Gamma(H)$ contains a diagonal connected component
isomorphic to $\Gamma(H)$. Let $\Gamma(H)\dot{\times}\Gamma(H)$ be
$\Gamma(H)\times\Gamma(H)$ with the diagonal component removed. 
Let $C=diam(\Gamma(H)\dot{\times}\Gamma(H))+1$. (Notice that $C$ is finite
by Definition~\ref{def_diam}.)

\begin{lema}\label{lm_length}
Let $w\in F$ and suppose that there exists a path $P(v,u,w)$ in 
$\Gamma(H)\dot{\times}\Gamma(H)$. Then $|w|_H\leq C$
\end{lema}
\begin{proof}
Take $a\in F$, $|a|<C$, such that there is a path 
$P(u,v,a)$ in $\Gamma(H)\dot{\times}\Gamma(H)$. It follows that
$P(v,v,wa)$ exist. So, $wa\in H^{v_1}\cap H^{v_2} \in \Omega(H)$ 
where $v=(v_1,v_2)$,
$v_1\neq v_2$. 
Now, 
$|w|_H=|waa^{-1}|_H\leq 1+|a|\leq C$.
\end{proof}
\begin{corollary}\label{cor_dpath}
Let $w\in F$, $|w|_H>C$. There exists at most one path with label $w$ in $\Gamma(H)$.
\end{corollary}
\begin{corollary}\label{cor_length}
Let $N\triangleleft H$ and $M$ be an $N$-diagram. 
Let $P$ be an essential piece in $M$.
Then $|l(P)|_H\leq C$. 
\end{corollary} 
\begin{proof} 
$f(\phi_D(P))\neq f(\phi_E(P))$ for all $\phi_D,\,\phi_E$ of Definition~\ref{def_diag}. 
Fix such $\phi_D,\;\phi_E$ ($\phi_D=\phi_E$ if $D=E$).
It follows that
in $\Gamma(H)$ there are two different paths with label $l(P)$. 
They define a path in
$\Gamma(H)\dot{\times}\Gamma(H)$ and the proof is concluded by Lemma~\ref{lm_length}.
\end{proof}

\begin{lema}\label{9}
Let $N\vartriangleleft H$. 
If a path $P(v,v,w)$ in $\Gamma (N)$ is cyclically reduced 
then $|w|_H \geq\gamma_H (N)-2diam(\Gamma(H))$.
\end{lema}
\begin{proof} 
Let $f:\Gamma(N)\to\Gamma(H)$ be a covering map, $u=f(v)$. There is a path
$P(1_H,u,b)$ in $\Gamma(H)$ with $|b|_H\leq diam(\Gamma(H))$. 
Let $Q$ be the lift of
$P(1_H,u,b)$ with $t(Q)=v$. The cycle $QP_wQ^{-1}$ shows that $bwb^{-1}\in N$, but
$\gamma_H(N)\leq |bwb^{-1}|_H\leq |w|_H+2|b|$.  
\end{proof}

\subsection{Proof of Theorem~\ref{th_acep21}}\label{proof_sub3}

Here we prove a version of Theorem~\ref{th_acep21} with an explicit estimate.
\begin{theorem*}
Let $N\triangleleft H$, 
$C=diam(\Gamma(H)\dot{\times}\Gamma(H))+1$ and $\gamma_H(N)\geq 6C+2diam(\Gamma(H))$.
Then $N=\langle\langle N\rangle\rangle_F\cap H$.
\end{theorem*}
Let $\Sigma=(\left\langle \left\langle N\right\rangle \right\rangle _F\cap H)\backslash N$.
In order to prove the theorem it suffices to show that $\Sigma=\emptyset$
under conditions of the theorem.

\begin{lema}\label{7}
Let $N\vartriangleleft H$ and 
$\Sigma =(\left\langle \left\langle N\right\rangle \right\rangle _F\cap H)\backslash N$. 
If $w \in \Sigma$, $\tau \in N$ and $h\in H$ then $\tau w  \in \Sigma$ and 
$w^h\in \Sigma $.
\end{lema}

\begin{proof}
As $N<\left\langle \left\langle N\right\rangle \right\rangle _F$ and $N\vartriangleleft H$, $\tau w \in \left\langle \left\langle N\right\rangle \right\rangle _F$ and $\tau w  \in H$. Suppose that $\tau w  \in N$, then $w \in \tau^{-1}N=N$, but $w \notin N$, it's a contradiction.
For the second part, it's clear that $w^h \in \left\langle \left\langle N\right\rangle \right\rangle_F \cap H$. Suppose $w^h\in N$, then $w\in hNh^{-1}=N$, but $w\notin N$, it's a contradiction.
\end{proof}
\begin{proposition}\label{prop_com_red}
Let $w_1,w_2\in F$ and $w_1$ be cyclically reduced. Suppose, that $w_1w_2=w_2w_1$ and there is a cycle 
$P(v,v,w_1)$ in $\Gamma(H)$. Then a path $P(v,v',w_2)$ and a cycle $P(v',v',w_1)$ exist in $\Gamma(H)$.
\end{proposition}
\begin{proof}
Notice that $w_1$ and $w_2$ are powers of the same element. The result follows by uniqueness of a path with
a given origin and label in a folded graph.
\end{proof}
\begin{lema}\label{lm_path_conj}
Let $aw^*a^{-1}\in H$ and $w^*$ be cyclically reduced. Then a path 
$P(1_H,v,a)$ and a cycle $P(v,v,w^*)$ exist in $\Gamma(H)$.
\end{lema}
\begin{proof}
If the products in $aw^*a^{-1}$ are without cancellation, then the conclusion of the lemma
is valid by definition of $\Gamma(H)$. Otherwise, there is a presentation 
$w^*=\alpha_1\alpha_2$ and  $aw^*a^{-1}=a'\alpha_2\alpha_1a'^{-1}$ such that
all products in $\alpha_1\alpha_2$ and $a'\alpha_2\alpha_1a'^{-1}$ are without 
cancellation. Then there are the following paths in $\Gamma(H)$:
$$
P(1_H,1_H,a'\alpha_2\alpha_1a'^{-1})=P(1_H,v',a')P(v',v,\alpha_2)P(v,v',\alpha_1)
P^{-1}(1_H,v',a').
$$
So, there is a path $P(1_H,v,a'\alpha_1^{-1})$ as a reduction of $P(1_H,v',a')P^{-1}(v,v',\alpha_1)$ and
a path $P(v,v,w^*)=P(v,v',\alpha_1)P(v',v,\alpha_2)$. Notice that $a=a'\alpha_1^{-1}\beta$ where $\beta \in F$
commute with $w^*$. The lemma follows by Proposition~\ref{prop_com_red}.    
\end{proof}
Suppose that $\Sigma\neq\emptyset$. Choose $u\in\Sigma$ with 
$r(u)=\min\{r(y)\;|\;y\in\Sigma\}$ and an optimal derivation 
$N$-diagram $M$ for $u$.
If $r(u)=1$ then $u=awa^{-1}\in H$, with $u\in H,\;w\in N$ and 
$a\in F\backslash H$. One can write $w=xw^*x^{-1}$ and $u=yw^*y^{-1}$ with
a cyclically reduced $w^*$. By Lemma~\ref{lm_path_conj}, there are paths
$P(1_H,v_x,x)$, $P(1_H,v_y,y)$ and cycles $P(v_x,v_x,w^*)$, $P(v_y,v_y,w^*)$.
Moreover, $v_x\neq v_y$ by Proposition~\ref{prop_paths}. It follows
that $w^*\in H^{v_x}\cap H^{v_y}$ and $|w^*|_H=1$. Take $z\in F$, $|z|\leq diam(\Gamma(H))$, 
such that a path $P(1_H,v_x,z)$ exist in $\Gamma(H)$. Then
$zx^{-1}\in H$ and $w'=zw^*z^{-1}\in N$. But $|w'|_H\leq 1+2|z|<\gamma(N)$ would be a contradiction. 

Therefore, suppose that $r(u)\geq 2$ and $M$ has more than one region.
There are two possibilities: 
\begin{enumerate}
\item All pieces are essential.
\item There is a inessential piece.
\end{enumerate} 
In the first case, by Lemma~\ref{9} and Corollary~\ref{cor_length} 
the number of pieces in $\delta(D)$ for a region $D$ is, at least, $6$ if 
$\delta(D)\cap\delta(M)$ does not contain an edge. So, $M$ is a (3,6)-diagram,
\cite{2}, (chapter 5). 
Theorem 4.3 of \cite{2}
is applicable because, for any region $D$, its boundary $\delta(D)$ does 
not contain vertices of degree $1$. It implies that
there is a region $D$ of $M$ such that 
$\delta (M) \cap\delta (D)$ contains a subpath $Q$ of $\delta(D)$ with 
$$
|l(Q)|_H \geq 3C>C
$$

Let $v_0\in \delta(M)$ be a base vertex of $\delta(M)$ for $u$. 
Let $v_1=o(Q)$, $v_2=t(Q)$. Let $P_D$ be a path in $\delta(D)$, such that 
$QP_D$ is a cycle $\delta(D)$. 
Let $P_1$ and $P_2$ be paths in $\delta(M)$ such that $o(P_1)=v_0$, $t(P_1)=v_1$,
$o(P_2)=v_2$, $t(P_2)=v_0$ and $u=l(P_1)l(Q)l(P_2)$. Denote
$l(P_1)=x_1$, $l(P_2)=x_2$, $l(Q)=y$, $l(P_D)=z$. 
In $\Gamma(N)$ there is a cycle $A$ with the label $yz$, its image $f(A)$ is 
a cycle in $\Gamma(H)$ with the same label. Also in $\Gamma(H)$ there are 
paths $P(1_H,\tilde v_1,x_1)$, $P(\tilde v_1, \tilde v_2, y)$ and $P(\tilde v_2,1_H,x_2)$, 
see figure~\ref{fig3}.
A path in $\Gamma(H)$ with label $y$ is unique by Corollary~\ref{cor_dpath}. So,
$P(\tilde v_1, \tilde v_2,y)$ is a subpath of $f(A)$. It follows that the 
path $P(1_H,\tilde v_1,x_1)$ has a lift $B$ on $\Gamma(N)$ such that $BAB^{-1}$ is a cycle
in $\Gamma(N)$. So, $x_1yzx_1^{-1}\in N$ by Proposition~\ref{8} and 
$u'=x_1z^{-1}x_2=(x_1z^{-1}y^{-1}x_1^{-1})(x_1yx_2)\in\Sigma$ by Lemma~\ref{7}. 
But $u'$ has a derivation $N$-diagram with $r(u)-1$ regions. 
This contradiction with minimality of $r(u)$ refutes the first case.

\begin{figure}[H] 
	\centering
	\includegraphics[scale=0.6]{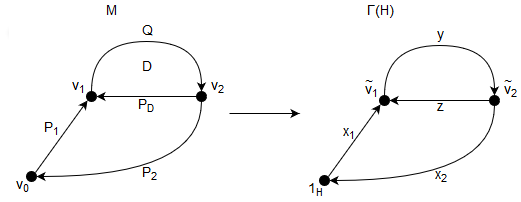}
      \caption{$M$ and $\Gamma (H)$}
      \label{fig3}
\end{figure}

Consider the second case. By Lemma~\ref{lm_ess_diff} an inessential piece 
$P$ is a path in the self-boundary of a region $D$. So, $\delta(D)=Q_1PQ_2P^{-1}$.
$Q_1$ and $Q_2$ define closed curves in $M$. Let $R_1$ and $R_2$ be compact subsets of 
the plane, bounded by $Q_1$ and $Q_2$ respectively. One has  $R_1\subseteq R_2$ or 
$R_2\subseteq R_1$. W.l.g. let $R_1\subseteq R_2$, see figure~\ref{fig4}. 
Denote $M_1=R_1\cap M$.
Notice that $M_1$ is a derivation $N$-diagram for $l(Q_1)$. 
Fix $\phi_D:\delta'(D)\to\Gamma(N)$ of Definition~\ref{def_diag}.
The piece $P$  splits in $P_1$ and $P_2$ in
$\delta'(D)$, see figure~\ref{fig4}.  
$f(\phi_D(P_1))=f(\phi_D(P_2))$ by 
definition of inessential boundary. Let $v=f(\phi_D(o(P)))\in V.\Gamma(H)$.
Take a path ${\cal T}$ from $1_H$ to $v$ in $\Gamma(H)$.  
Add a path (``tail'') $T$ with $l(T)=l({\cal T})$ to
the diagram $M_1$. The resulting diagram $M_1'$ is a derivation $N$-diagram
for $l(TQ_1T^{-1})=u'\in H$. 
There are two possibilities:
\begin{description}
\item{i)} $u'\in\Sigma$
\item{ii)} $u'\in N$.
\end{description} 
In the case {\bf i)} $r(u')<r(u)$ and we get a contradiction with the minimality of  $r(u)$.
Consider case {\bf ii)}.  Notice that $l(TQ_1PQ_2P^{-1}T^{-1})\in N$ and, consequently,   
$l(TPQ_2P^{-1}T^{-1})\in N$.
So, there is
a cycle in $\Gamma(N)$ with label $l(Q_2)$. It follows that we may remove $D$ and $R_1$,
and glue new region $D'$ along $Q_2$. A new diagram $M'$ is a derivation $N$-diagram 
for $u$ with $r(M')\leq r(M)-1$, a contradiction.

\begin{figure}[H]
	\centering
	\includegraphics[scale=0.5]{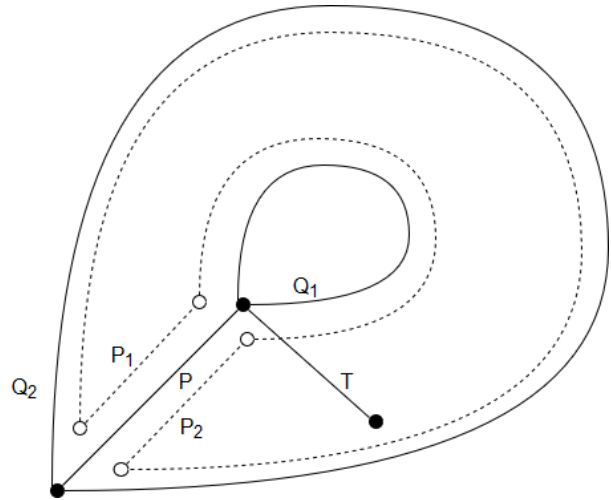}
	\caption{Region with self-boundary}
        \label{fig4}
\end{figure}

\subsection{Another proof of Theorem~\ref{th_acep21}}

Theorem~\ref{th_acep21} may be seen as a particular case of
Theorem 7.15, \cite{13}. For proof, we need the adapted definition of
the distance $\hat d$ of \cite{13} only. We do not state the results of \cite{13},
the interested reader might choose to look at \cite{13}.   

A product $w_1h_1\dots h_{n-1}w_n$ is said
to be alternating if
\begin{description}
\item[I)] $h_i\in H$\;
\item[II)] $w_1h_1\dots w_i\not\in H$ for $i<n$.
\end{description} 
The weight of an alternating product is
$$
\|w_1h_1\dots h_{n-1}w_n\|_H=\sum_{i=1}^n(|w_i|+1)-1
$$
For $h\in H$ define $\hat d(1,h)=\hat d_H(1,h)$ to be the least $\|w_1h_1\dots h_{n-1}w_n\|_H$
over all alternating products $w_1h_1\dots h_{n-1}w_n=h$.
\begin{proposition}\label{prop_lip_equiv}
The length functions $\hat d_H(1,\cdot)$ and $|\cdot|_H$ are Lipshitz equivalent.
\end{proposition} 
Theorem~7.15 of \cite{13} implies an analogue of our Theorem~\ref{th_acep21},
with the length function $|\cdot|_H$ replaced by the length function $\hat d_H(1,\cdot)$,
see \cite{13} for details. So,  Theorem~\ref{th_acep21} follows by Proposition~\ref{prop_lip_equiv}.
The drawback of this approach is the absence of an estimate for $\gamma_H$.   
\subsubsection{Proof of Proposition~\ref{prop_lip_equiv}}
Let a free group $F$ be  freely generated by $X$.
 Let $w=x_1x_2\dots x_n$ be a reduced word in 
$X\cup X^{-1}$. We call a prefix $x_1x_2\dots x_i$ admissible if $P(1_H,v,x_1\dots x_i)$ exist for some $v\in V.\Gamma(H)$ (such a path is unique).
A prefix $x_1\dots x_i$ is a maximal admissible prefix if it is admissible and
$x_1\dots x_{i+1}$ is not (or $i=n$).

Fix an alternating product $w_1h_1\dots w_{n-1}h_{n-1}w_n$.
Let $a_i\circ b_i=w_1h_1\dots w_ih_i$, where $a_i\circ b_i$ is a product without
cancellation and $a_i$ is a maximal admissible prefix of $a_ib_i$.
\begin{lema}\label{lm_admissible}
$|a_i|_H\leq \sum\limits_{j=1}^{i}(|w_j|+C)$, where $C$ is of 
Lemma~\ref{lm_length}. 
\end{lema}     
\begin{proof}
Induction on $i$.
For $i=0$ there is nothing to prove.

$i\Longrightarrow i+1$. Let $P(1_H,v_i,a_i)$ be a path in $\Gamma(H)$ and
$a_{i+1}\circ b_{i+1}=a_ib_iw_{i+1}h_{i+1}$.
If not all letters of $b_i$ cancel with $w_{i+1}h_{i+1}$ then $a_{i+1}=a_i$ and
proof is concluded by induction. So, let $a_{i+1}=a_i\alpha$ with a path 
$P(v_i,v_{i+1},\alpha)$. (Notice that cancellations in $a_i\alpha$ do not cause 
problems.) 
We may represent $\alpha=\alpha_0\circ\alpha_1$,
where $\alpha_0$ is a subword of $w_{i+1}$ and $\alpha_1$ is a subword of 
$h_{i+1}$. (Some $\alpha_i$ may be empty.) We may decompose 
$P(v_i,v_{i+1},\alpha_0\circ\alpha_1)=
P(v_i,v',\alpha_0)P(v',v_{i+1},\alpha_1)$.
As $\alpha_1$ is a subword of $h_{i+1}$ there is $P(u',u,\alpha_1)$ which is
a subpath of $P(1_H,1_H,h_{i+1})$, i.e.. 
$$
P(1_H,1_H,h_{i+1})=P(1_H,u',h'_{i+1})P(u',u,\alpha_1)
P(u,1_H,h''_{i+1}),
$$
for corresponding representation $h_{i+1}=h'_{i+1}\circ\alpha_1\circ h''_{i+1}$.
We claim that $u'\neq v'$. Indeed, 
$a_{i+1}=w_1h_1...w_ih_iw_{i+1}h'_{i+1}\alpha_1$ and there is a path 
$P(1_H,v_{i+1},a_{i+1})$. If $u'=v'$ then $v_{i+1}=u$ by the uniqueness 
of the path with given origin and label. So, the reduction of a path 
$P(1_H,v_{i+1},a_{i+1})P^{-1}(1_H,u=v_{i+1},h'_{i+1}\alpha_1)$ is 
a path $P(1_H,1_H,w_1h_1\dots w_ih_iw_{i+1})$, a contradiction with
{\bf II)}.
So, $u'\neq v'$ and there is a path in $\Gamma(H)\dot{\times}\Gamma(H)$ with label 
$\alpha_1$. It follows that  
 $|\alpha_1|_H\leq C$ by Lemma~\ref{lm_length}. It
is also clear that $|\alpha_0|\leq |w_{i+1}|$.  
\end{proof}
Taking alternating representation $h=w_1h_1w_2h_2\dots w_n$ for $h\in H$ and substituting $i=n$ in
Lemma~\ref{lm_admissible} we get the following corollary.
\begin{corollary}\label{cor_leq1}
$|h|_H\leq C\hat d_H(1,h)$ for  $h\in H$.
\end{corollary}

Consider the other inequality. 
\begin{lema}\label{lm_leq2}
$\hat d_H(1,h)\leq (1+2diam(\Gamma))|h|_H$.
\end{lema}
\begin{proof}
Let $|h|_H=k$. Then there is a presentation $h=w_1h_1w_2h_2\dots h_{n-1}w_n$ with 
$1\neq h_i\in H^{a_i}\cap H^{a'_i}$ and
$k=\sum\limits_{i=1}^n(|w_i|+1)-1$.
W.l.g. $|a'_i|,\;|a_i|\leq diam(\Gamma(H))$.
Define an alternating product 
$h=w'_1h'_1\dots w'_n$ where $h'_i=b_ih_ib_i^{-1}$ and $w'_i=b_{i-1}w_ib_{i}^{-1}$,
$b_0=b_n=1$. Here $b_i=a_i$ or $a'_i$ is chosen in such a way that 
$w'_1h'_1\dots w'_i\not\in H$ for $i<n$.
Clearly, 
$|w'_{i}|\leq |w_i|+2diam(\Gamma)$. 
\end{proof}
Notice that Corollary~\ref{cor_leq1} and Lemma~\ref{lm_leq2} imply Proposition~\ref{prop_lip_equiv}.
\bibliographystyle{ieeetr}
\bibliography{bib}

\end{document}